\documentclass[12pt]{amsart}
\usepackage{geometry}                
\usepackage{graphicx}
\usepackage{enumerate}
\usepackage{amssymb}
\usepackage{epstopdf}
\usepackage{color}

\newtheorem{lemma}{Lemma}[section]

\newtheorem{theorem}[lemma]{Theorem}
\newtheorem{rem}[lemma]{Remark}

\def\R{{\relax\ifmmode I\!\!R\else$I\!\!R$\fi}}
\newcommand{\beqn}{\begin{equation}}
\newcommand{\eeqn}{\end{equation}}

\newcommand\eref[1]{(\ref{#1})}
\newcommand{\e}{\epsilon}

\newcommand{\al}{\alpha}
\newcommand{\be}{\beta}

\newcommand{\sgn}{\mathrm{sgn}}

\newcommand{\argmin}{\mathrm{argmin}}
\newcommand{\argmax}{\mathrm{argmax}}

\newcommand{\ran}{\rangle}
\newcommand{\lan}{\langle}
\newcommand{\abs}[1]{\lvert#1\rvert}
\providecommand{\norm}[1]{\lVert#1\rVert}

\title{Greedy Strategies for Convex Optimization}
\author{ 
Hao Nguyen and   Guergana Petrova}
\thanks{%
    This research was supported by the Office of Naval Research Contract    ONR N00014-11-1-0712, by  the NSF Grant   DMS 1222715,  and by the Bulgarian Science Fund Grant DFNI-T01/0001.  }
\begin{document}
\maketitle
\begin{abstract}
We investigate two greedy strategies for finding an approximation to the
minimum  of a convex function $E$ defined on a Hilbert space $H$. We prove convergence rates for these
algorithms under suitable conditions on the objective function $E$. These conditions involve the behavior of the modulus of 
smoothness and the modulus of uniform convexity of $E$.

\noindent
\noindent
{\bf Key Words:} 
Greedy Algorithms, Convex Optimization, Rates of Convergence.\\

 \end{abstract}

\section{Introduction}
\label{Intr}
Convex optimization has  many application domains  such as automatic control systems,  signal processing, communications and networks, electronic circuit design, data analysis and modeling, statistical estimation, finance, and combinatorial optimization.   
A general description for  convex optimization is that we are given a Banach space $X$ and a convex function $E$ on $X$ whose minimum we wish to compute.      Thus, we are interested in the  development and analysis of  algorithms  for approximating 
\begin{equation}
\label{Eq}
  \inf_{x\in D} E(x),
\end{equation}
where   $D$ is a convex subset of $X$.  $E$ is called the {\it objective} function and,  by the convexity assumption, satisfies the   condition
$$
E(\gamma x+\delta y)\leq \gamma E(x)+\delta E(y), \quad x,y\in D, \quad \gamma, \delta\geq 0, \quad \gamma+\delta=1.
$$
The classical results on convex optimization deal with objective functions $E$ defined on subsets in $\R^d$  with moderate values of $d$, see e.g. \cite{BV}.    However, several of the 
 applications,  listed above, lead to optimization on   Banach spaces   of dimension $d$,  where $d$ is quite large or even $\infty$.  The design of algorithms for such high dimensional problems is quite challenging, typical
convergent results involve  the dimension $d$ and suffer from the curse of dimensionality. 

Recently,  several researchers (see e.g. \cite{Temlyakov1,Temlyakov2, Z}), have proposed 
   strategies for solving (\ref{Eq}), where   the curse of dimensionality is overcome by using  greedy techniques, similar to those originally
developed for the approximation of a given element $x\in X$.
The minimum in \eref{Eq} is approximated by $E(x_m)$, $m=0,1,\dots$, where each $x_m$    is constructed as a linear combination of $m$ elements (i.e. $x_m$ is $m$ {\it sparse}) from a given  dictionary ${\mathcal D}$.  Recall that ${\mathcal D}$ is called a symmetric dictionary if each $\varphi \in {\mathcal D}$ has norm $\|\varphi\|\leq 1$, if $\varphi \in {\mathcal D}$, then $-\varphi\in {\mathcal D}$, and the closure of $span\, {\mathcal D}$ is $X$.  
A typical a priori convergence result given by the above authors for these greedy algorithms is  proven under two assumptions:
\vskip .1in
\noindent
{ (i)}  An assumption  on the smoothness of $E$.
\vskip .1in
\noindent
{ (ii)} An assumption that the minimum in \eref{Eq} is taken at a point $\bar x$ which  is in the convex hull of the dictionary ${\mathcal D}$.
\vskip .1in

In this paper, we investigate  the special case when $X=H$ is a Hilbert space, the dictionary ${\mathcal D}=\{\pm \varphi_j\}_{j=1}^\infty $,  where $\{\pm \varphi_j\}_{j=1}^\infty $is an orthonormal basis,  and $D=H$ (which corresponds to  global minimization).   We assume that the global minimum is attained at some point $\bar x\in H$.  It follows then that the minimum is taken on  the set 
$$\Omega:=\{x\in H:\,\,E(x)\leq E(0)\}.$$ 
We assume throughout this paper that the set $\Omega$ is bounded in $H$.
We impose the following assumptions on the  objective  function  $E$: 

\vskip .1in
\noindent
{\bf Condition 0:}
 $E$   has a  Frechet derivative $E'(x)\in H$ at each point $x$ in $\Omega$ and  
  $$\|E'(x)\|\le M_0,\quad x\in \Omega,$$
  where throughout $\|\cdot\|$ denotes the norm on $H$.
 
\vskip .1in
\noindent
{\bf Condition 1:} There are constants $0< \al$, $1<q\leq 2$ and $0<M$, such that 
for all $x$, $x'$ with $\|x-x'\|\leq M$, $x\in \Omega$,
\begin{equation}
\label {cond1}
 E(x')-E(x)-\langle E' (x),x'-x\rangle \leq \al \|x'-x\|^q.
\end{equation}

\vskip .1in
\noindent
{\bf Condition 2:}
There are constants $0< \be$, $2\leq p<\infty$ and $0<M$, such that for all $x$, $x'$ with $\|x-x'\|\leq M$, $x\in \Omega$,
\begin{equation}
\label {cond2}
 E(x')-E(x)-\langle E' (x),x'-x\rangle \geq \be \|x'-x\|^p.
\end{equation}
 \vskip .1in
 We show in \S\ref{cond} that {\bf Condition 1}  is equivalent to conditions on the modulus of smoothness $\rho(E,u)$, and  {\bf Condition 2}  is equivalent to conditions on the modulus of uniform convexity $\delta_1(E,u)$, as  usually defined
 in convex optimization (see e.g. \cite{Za}), and introduced by us in \S\ref{cond}.
 
 We study two greedy procedures for solving \eref{Eq}.  The first is  the analogue for convex minimization of the Orthogonal Matching Pursuit Algorithm
used  for approximation (see \cite{MZ}).    We  denote this convex minimization algorithm by OMP(co)\footnote{Here and later we will use the abbreviation (co) if an algorithm is used for convex optimization}.     The second is the  Weak Chebyshev Greedy Algorithm (WCGA(co)) as introduced by Temlyakov \cite{Temlyakov2}.  These greedy procedures, which are defined in \S \ref{greedy}, iteratively generate a sequence $x_m$, $m=0,1,\dots$,
where each $x_m$ is $m$ sparse, and then use $E(x_m)$ as the approximation to the minimum $E(\bar x)$.

 Our main results are Theorem \ref{Tm0} and  Theorem \ref{wlm0}  which establish a priori convergence rates for both  OMP(co) and the  WCGA(co) when  they are used to find the minimum of a function $E$ that satisfies {\bf Conditions 0},  {\bf 1} and  {\bf 2}. For example, we show that if  the objective function $E$ satisfies {\bf Condition 0} and {\bf Condition  1}, is strongly convex on $H$ (therefore satisfies {\bf Condition 2} with $p=2$), and its minimizer  $\bar x$ is sparse with respect to ${\mathcal D}$, then the error at the $m$-th step of the OMP(co)  satisfies the inequality
$$
E(x_m)-E(\bar x)\leq C_0m^{1-\frac{q}{2-q}}, \quad 1<q<2,
$$
and 
$$
\|x_m-\bar x\|\leq C_1m^{\frac{1}{2}-\frac{q}{2(2-q)}},
$$
where $C_0=C_0(q,E)$ and $C_1=C_1(q,E)$.   We also prove exponential convergence  in the case   $q=2$.   
  In contrast, the results from \cite{Temlyakov2} and \cite{Z} do not impose {\bf Condition 2} and   only give the rate $1-q$ . In summary, we show that imposing more conditions on the convexity of  the objective function $E$ (like {\bf Condition 2}) results in provably improved convergence rates for both OMP(co)  and WCGA(co).

 \section{Conditions on $E$}
 \label{cond}

 In this section, we discuss the compatibility of the conditions ({\bf Condition 0} , {\bf Condition 1} and {\bf Condition 2}) imposed on the objective function $E$
 and their relation to the modulus of smoothness and modulus of uniform convexity of $E$.
  We  recall that a function $E$ is Frechet differentiable at $x\in \Omega$ if
  there exists a bounded linear functional, denoted by $E'(x)$,  such that
$$
\lim_{h \to 0} \frac {|E( x+h)-E( x)-\lan E'(x),h\ran|}{\|h\|} =0.
$$

 We start with  discussing the connection between {\bf Condition 1} and the  modulus of uniform smoothness of $E$ on $\Omega$.

\subsection{Condition 1}

\noindent

 Given a convex function $E: H \to \R$ and a set $S \subset H$, the modulus of smoothness of $E$ on $S$ is defined by
 \begin{equation}
 \label{mos}
  \rho(E,u) :=\rho(E,u,S):=\frac 12 \sup_{ x\in S,\|y\|=1}\left\{E(x+uy)+E(x-uy)-2E(x)\right\} ,\quad u>0,
 \end{equation}
 and  the modulus of uniform smoothness of $E$ on $S$ is defined by
\begin{equation}
 \label{mus}
\rho_{1}(E,u,S):=\sup_{x\in S, \|y\|=1, \lambda \in (0,1)}\left\{\frac{(1-\lambda)E(x-\lambda uy)+\lambda E(x+(1-\lambda) uy)-E(x)}{\lambda(1-\lambda)}\right\}.
 \end{equation}

These  two moduli of smoothness  are equivalent (see  \cite{Za}, page 205):
\begin{lemma}
\label{equimodu}
Let $E$ be a convex function defined on $H$, and let $S\subset H$, then
\begin{equation}
4\rho( E,\frac u 2, S) \leq \rho_{1}(E,u,S) \leq 2 \rho(E,u,S).
\end{equation}
\end{lemma}

The next lemma shows  the relation between the modulus of uniform smoothness and {\bf Condition 1}. 
\begin{lemma}
\label{uniformsmooth}
Let $E$ be a convex function defined on a Hilbert space $H$ and $E$ be Frechet differentiable on a set $S \subset H$. The following statements are equivalent for any $q \in (1,2]$ and $M>0$. 
\vskip .1in
\noindent
{\rm (i)}  There exists  $\al > 0$, such that
 for all  $x \in S, x' \in H, \|x-x'\|\leq M$, 
\begin{equation}
\label{unismooth1}
E(x')-E(x)-\lan E'(x),x'-x \ran \leq \al \|x'-x\|^{q}.
\end{equation}
\vskip .1in
\noindent 
{\rm (ii)}There exists $\al_{1} > 0$,  such that
\begin{equation}
\label{unismooth2}
\rho(E,u,S) \leq \al_{1} u^{q}, \quad 0<u \leq M.
\end{equation}
The same result holds with $\rho $ replaced by $\rho_1$.
\end{lemma}

\begin{proof}  While this is a particular case of Corollary 3.5.7 from \cite{Za},  for completeness of this paper, we provide a simple proof of this lemma.     First, observe that because of Lemma \ref{equimodu}, statement (ii) for $\rho$ and $\rho_1$ are equivalent, and so we can use them interchangeably.
Assume that the first statement  is true. 
For any $x \in S$, $y\in H$,  $\|y\|=1$ and any $0<u\leq M$, let $x':=x+uy$,  $x'':=x-uy$. Then,  we have $\|x-x'\|=u\leq M$,  $\|x''-x\|=u \leq M$. We apply (\ref{unismooth1}) for the pairs $(x',x) $ and $(x'',x)$ to obtain
$$E(x+uy)-E(x)-u\lan E'(x),y \ran \leq  \al u^{q},\quad 
E(x-uy)-E(x)+u\lan E'(x),y\ran \leq \al u^{q}.$$
Therefore, we have
$$E(x+uy)+E(x-uy)-2E(x) \leq 2\al u^{q}.
$$
We take the supremum over $x \in S, y \in H, \|y\|=1$ and derive
$\rho(E,u,S) \leq \al u^{q}$, $0< u \leq M$, which gives the lemma for $\rho$.
 
Conversely, suppose  that (ii) holds for $\rho_1$.   Then,
 for any $\lambda \in (0,1)$ and any $x \in S, y\in H, \|y\|=1, 0<u\leq M$, 
$$
\frac{(1-\lambda)E(x-\lambda uy)+\lambda E(x+(1-\lambda) uy)-E(x)}{\lambda(1-\lambda)}
\leq \al_{1}u^{q}.
$$
This is the same as saying
$$
\frac{E(x-\lambda u y)-E(x)}{(1-\lambda)\lambda}+\frac{E(x+(1-\lambda)uy)-E(x-\lambda u y)}{1-\lambda} \leq \al_{1}u^{q}.
$$
We let $\lambda \to 0^{+}$ and use  the continuity of $E$ and the definition of Frechet derivative $E'(x)$ with $h=-\lambda u y$, to obtain
$$
\lan E'(x),-uy\ran+E(x+uy)-E(x) \leq \al_{1}u^{q}.
$$
Now, for any $x\in S$,  $x' \in H$,  $\|x'-x\| \leq M$, we let $u=\|x'-x\|$,  $y=\frac{x'-x}{\|x'-x\|}$. 
The above inequality can be written as
$$
E(x')-E(x)-\lan E'(x), x'-x\ran \leq \al_{1}\|x'-x\|^{q},$$
which is (\ref{unismooth1}) with $\al=\al_{1}$. 
\end{proof}
\noindent

\subsection{Condition 2}

\noindent

We first observe the following:

\noindent
{\bf Claim 1.}  
{\it If {\bf Condition 2} holds for a convex function $E$  and a set $\Omega$ that is convex and bounded, then 
{\bf Condition 2} holds for all $x$, $x'\in \Omega$  with $\beta$ replaced by $\beta_0>0$. }
\vskip .1in
{\it Proof.} 
  Since $\Omega$ is bounded, there is $L>0$, such
that $diam(\Omega)\leq LM$. Let $x,x'\in \Omega$. If $\|x-x'\|\leq M$, {\bf Condition 2} holds for the pair $(x,x')$ provided $\beta_0\le \beta$.   If $\|x-x'\|>M$, we 
chose a point $x_1$, such that
$$
x_1=\gamma x'+(1-\gamma)x\in \Omega, \quad \gamma:=\frac{M}{\|x-x'\|}\geq L^{-1}.
$$
Clearly $\|x-x_1\|=M$, and therefore 
$$E(x_1)-E(x)-\langle E' (x),x_1-x\rangle \geq \be  \|x_1-x\|^p.$$
 Because of the convexity of $E$, 
$$
 E(x_1)-E(x)\leq \gamma [E(x')-E(x)].
$$
A combination of the last two inequalities  and the fact that 
$x_1-x=\gamma(x'-x)$
result in
$$
 E(x')-E(x)-\langle E' (x),x'-x\rangle \geq \be \gamma^{p-1}\|x'-x\|^p\geq \beta L^{1-p}\|x'-x\|^p.
$$
Therefore, the claim has been proven with $\beta_0=\min\{\beta  ,\beta L^{1-p}\}$.
\hfill $\Box$

 \vskip .1in 
Note that {\bf Condition 2} is a generalization of the notion of strongly convex functions.
Recall that a  function $E$ is called strongly convex on $H$, if there is a constant $\beta>0$, called the convexity parameter of  $E$, such that
$$
E(x')-E(x)-\langle E' (x),x'-x\rangle \geq \be \|x'-x\|^2 ,\quad x,x'\in H.
$$

Next, we discuss the  compatibility  between the convexity of $E$ and {\bf Condition 2}.
 
\begin{lemma}
\label{con}
Let $E$ be a Frechet differentiable function on $H$. 
$E$ is convex on H if and only if 
$$
 E(x')-E(x)-\langle E' (x),x'-x\rangle \geq 0, \quad \hbox{for all}\,\, x,\,\,x'\in H, \quad \|x-x'\|\leq M.
 $$
\end{lemma}
\noindent
{\it Proof.} For convex functions on $\R^n$, a proof (without the restriction $\|x'-x\|\le M$)  can be found in \cite{BV}.  Simple modifications of this proof (which we do not give) result in a proof of the lemma.  \hfill $\Box$

Finally, we present a concept which is dual to the modulus of uniform smoothness for convex functions, called the modulus of uniform convexity (see \cite{BV2,Za}) and show how it is related to {\bf Condition 2}.
Given a convex function $E: H \to \R$ and a set $S\subset H$, its modulus of uniform convexity on $S$ is defined by 
\begin{equation}
\delta_{1}(E,u,S):=\inf_{x\in S, \|y\|=1, \lambda \in (0,1)}\left\{\frac{(1-\lambda)E(x-\lambda uy)+\lambda E(x+(1-\lambda) uy)-E(x)}{\lambda(1-\lambda)}\right\}.
\end{equation}
We prove a lemma (see \cite{Za}) that shows the equivalence of  {\bf Condition 2} and certain behavior of the
modulus of uniform convexity $\delta_{1}$ of $E$.
\begin{lemma}
\label{uniformconvex}
Let $E$ be a convex function defined on a Hilbert space $H$ and $E$ be Frechet differentiable on $S \subset H$. The following statements are equivalent for any $p \in [2,\infty)$ and $M>0$. 

 \noindent
{\rm (i)} There exists $\be > 0$, such that
 for all $x \in S,\,\,x'\in H, \|x-x'\|\leq M,$
\begin{equation}
\label{st3}
E(x')-E(x)-\lan E'(x),x'-x \ran \geq \be \|x'-x\|^{p}.
\end{equation}

\noindent
{\rm (ii)} There exists  $\be_1 > 0$,  such that
\begin{equation}
\label{st4}
\delta_{1}(E,u,S) \geq \be_1  u^{p}, \quad 0< u \leq M.
\end{equation}
\end{lemma}
\begin{proof}
Assume that the first statement  is true.
For any $x \in S$, $y\in H$,  $\|y\|=1$, $0<u\leq M$ and $\lambda \in (0,1)$, let $x':=x-\lambda uy,$  $x'':=x+(1-\lambda)uy$.  Then, we have $\|x-x'\|= \lambda u\leq M$,  $\|x''-x\| =(1-\lambda)u\leq M$. 
We apply (\ref{st3}) for $x\in S$,  $x'\in H$ and $x\in S$,  $x''\in H$ to derive
$$E(x-\lambda uy)-E(x)+\lambda u\lan E'(x),y \ran \geq  \be \lambda^{p}u^{p},$$
$$E(x+(1-\lambda)uy)-E(x)-(1-\lambda)u\lan E'(x),y\ran \geq \be(1-\lambda)^{p} u^{p}.$$
Multiplying the first inequality by $(1-\lambda)$,  the second one by $\lambda$ and adding them  yields
$$(1-\lambda)E(x-\lambda uy)+\lambda E(x+(1-\lambda)uy)-E(x) \geq \be\lambda (1-\lambda)(\lambda^{p-1}+(1-\lambda)^{p-1}) u^{p}.$$
Since  $\lambda^{p-1}+(1-\lambda)^{p-1} \geq 2^{2-p}$ for $\lambda \in (0,1)$, we have
$$
\frac{(1-\lambda)E(x-\lambda uy)+\lambda E(x+(1-\lambda)uy)-E(x)}{\lambda(1-\lambda)} \geq 2^{2-p}\be u^{p}.$$
We take the infimum over $x \in S$,  $y \in H$,  $\|y\|=1$ and $\lambda \in (0,1)$ and obtain that
$\delta_{1}(E,u,S) \geq 2^{2-p}\be u^{p}$, $0 <u \leq M$,
which is (\ref{st4}) with $\be_{1}=2^{2-p}\be$.

Conversely, suppose that  for some $\be>0$ we have $\delta_{1}(E,u,S) \geq \be u^{p}$ for all $0<u\leq M$.
It follows from the definition of $\delta_{1}$ that
 for any $\lambda \in (0,1)$, $x \in S$,  $y\in H$,  $\|y\|=1$ and $0<u\leq M$, 
$$
\frac{(1-\lambda)E(x-\lambda uy)+\lambda E(x+(1-\lambda) uy)-E(x)}{\lambda(1-\lambda)}
\geq \be_{1}u^{p}.
$$
This is the same as saying
$$
\frac{E(x-\lambda u y)-E(x)}{\lambda}+\frac{E(x+(1-\lambda)uy)-E(x)}{1-\lambda} \geq \be_1 u^{p}.
$$
We let $\lambda \to 0^{+}$ and by the continuity of $E$ and the definition of Frechet derivative $E'(x)$ for $h=-\lambda u y$, we obtain
$$
\lan E'(x),-uy\ran+E(x+uy)-E(x) \geq \be_1u^{p}.
$$
Now, for any $x\in S$,  $x' \in H$,  $\|x'-x\| \leq M$, we let $u=\|x'-x\|$,  $y=\frac{x'-x}{\|x'-x\|}$ and derive
$$
E(x')-E(x)-\lan E'(x), x'-x\ran \geq \be_1 \|x'-x\|^{p},
$$
which is (\ref{st3}) with $\beta=\beta_1$.
\end{proof}
\subsection{The conditions on $E$ and their connection  to Compressed Sensing.}
Let us summarize that as a result of  Lemma \ref{uniformsmooth}  and Lemma \ref{uniformconvex}, we have proven the following.
\begin{lemma}
Let $E$ be a convex function defined on  a Hilbert space $H$. 
Let us denote by $\Omega$ the set $\Omega=\{x \in H: E(x) \leq E(0)\}$ and $E$ be Frechet differentiable on $\Omega$. Let $\delta_1(E,\cdot,\Omega)$ and  $\rho_1(E,\cdot,\Omega)$ be the  modulus of uniform convexity and modulus of uniform smoothness of $E$ on $\Omega$, respectively. The following two statements are equivalent
\vskip .1in
\noindent
{\rm (i)} $E$ satisfies {\bf Condition 1} and {\bf Condition 2}.
 \vskip .1in
 \noindent
 {\rm (ii)} There exist constants $\al_1 > 0, \be_1 > 0$,
  such that
\begin{equation}
\nonumber
\be_1 u^p \leq \delta_1(E,u,\Omega) \leq \rho_1(E,u,\Omega) \leq \al_1  u^q, \quad u \in (0,M].
\end{equation}
 
\end{lemma}

Let us next observe that (i) of the above lemma has a similar flavor to conditions that are imposed in compressed sensing.
Indeed, conditions similar to {\bf Condition 1} and {\bf Condition 2} have been considered by  Zhang in   \cite{Z1}, where he solves a
sparse optimization problem in $\R^n$, using greedy based strategies.  He considers  any convex function $E$ on $\R^n$ for which there are constants $\al(s), \be(s) >0$ such that
\begin{equation}
\label {SRC}
\be(s) \|x'-x\|_{2}^2 \leq E(x')-E(x)-\langle E' (x),x'-x\rangle \leq \al(s) \|x'-x\|_{2}^2,
\end{equation}
holds whenever $x,x' \in \R^n$ and   $x-x'$ has $\leq s$ nonzero coordinates.
Notice, that \eref{SRC} is the same as our {\bf Condition 1} and {\bf Condition 2} except that it is only required to hold
whenever $x-x'$ is $s$ sparse whereas in our case we require this to hold for all $x,x'$ with $\|x-x'\|\le M$.
Zhang applied his results to the decoding problem in compressed sensing in which case $E(x)=\|Ax-b\|_2^2$,   and $A$ is a given $k\times n$ 
matrix with $k<<n$. For this choice of  $E$, the Frechet derivative
$E'(x)$ can be computed explicitly as
$\langle E'(x),. \rangle=2\langle A^T( Ax-b),.\rangle$.  Moreover, we have
   \begin{eqnarray*}
 E(x')-E(x)-\langle E' (x),x'-x\rangle&=&\|Ax'-b\|_2^2-\|Ax-b\|_2^2 -2 \langle A^T(Ax-b),x'-x \rangle\\
  &=&\|Ax'-Ax\|_2^2=\|A(x-x')\|_2^2.
 \end{eqnarray*}
If we  denote by $z=x'-x$, condition (\ref{SRC}) becomes
\begin{equation}
\nonumber
\be(s) \|z\|_{2}^2 \leq \| Az\|_{2}^2 \leq \al(s) \|z\|_{2}^2,
\end{equation}
for $s$ sparse vectors $z\in \R^n$. This condition  is known as the {\it Restricted Isometry Property}  and was first introduced by Candes and Tao (see \cite{CT1}, \cite{Dono}).  For  applications in compressed sensing one needs that $\alpha(s),\beta(s)$ are sufficiently close to one.

\section{Greedy algorithms for optimization}
\label{greedy}
In this section, we  introduce the two algorithms for convex minimization in a Hilbert space $H$ that we will analyze.
As usual, we assume that $\{\varphi_j\}_{j=1}^\infty$ is an orthonormal  basis for $H$.  We begin with the OMP(co) algorithm.
\vskip 0.1in

\noindent
 {\bf Orthogonal Matching Pursuit (OMP(co)): }
 \begin{itemize}
 \item  {\bf Step $0$}:    Define $x_0:=0$.  
 If $E'(x_0)=0$, stop the algorithm and define $x_k:=x_0$, $k\ge 1$.
 
 \item {\bf Step $ m$}:  Assuming $x_{m-1}$ has been defined and $E'(x_{m-1})\neq 0$,
 Find 
 $$\varphi_{j_m}:=\argmax\{ |\langle E'(x_{m-1}),\varphi \rangle|, \varphi \in \mathcal D\},$$
 and define 
  $$x_m:=\displaystyle {\argmin_{x\in span\{\varphi_{j_1}, \varphi_{j_2},\dots,\varphi_{j_m}\}}  E(x)}.$$
 If $E'(x_m)=0$,   stop the algorithm and define $x_ k:=x_m$, $k>m$.  Otherwise, go to {\bf Step $m+1$}.
         \end{itemize}   

 \vskip .1in

  Note that if the algorithm stops at step $m$, then the output $x_m$ of the algorithm is  the minimizer $\bar x$, because of the following well-known lemma.
  \begin{lemma}
\label{FT}
Let  $E$ be a Frechet differentiable convex function, defined on a convex domain $\Omega$.  
Then $E$ has a global minimum at $\bar x \in \Omega$ if and only if $E'(\bar x)=0$.
\end{lemma}

\vskip .1in
\noindent
 {\bf Weak Chebyshev Greedy Algorithm (WCGA(co)): }
The  description of the WCGA(co) is the same as the OMP(co), with the only  difference that a sequence $\{t_k\}_{k=1}^{\infty}$,  $t_k\in (0,1]$
is used to weaken the condition on the choice of $\varphi_{j_m}$. Namely, $\varphi_{j_m}$ is now chosen to satisfy the inequality
$$
|\langle E'(x_{m-1}),\varphi_{j_m} \rangle|\geq t_m\sup_{\varphi \in \mathcal D} \langle E'(x_{m-1}),\varphi\rangle.
$$
When all $t_k=1$, $k\ge 1$, the  WCGA(co) becomes the  OMP(co).

Let us remark that neither of these two algorithms generates a unique sequence $x_m$, $m\ge 0$.  The analysis
that follows applies to any   sequence generated by the corresponding algorithm.
  
For comparison with the results we prove in this paper, we recall the result of  Temlyakov.   Let  ${ A_1}({\mathcal D})$ denote the closure (in $H$) of the convex hull of ${\mathcal D}$.  The following theorem was proved  in \cite{Temlyakov2} in a more general setting of Banach spaces and  general symmetric dictionaries.
\begin{theorem}[\cite{Temlyakov2} Theorem 2.2]
\label{TE}
Let $E$ be a uniformly smooth convex function defined on a Banach space $X$
and  let the set $\Omega:=\{x:E(x) \leq E(0)\}$ be bounded.
Let the  modulus of smoothness of $E$ on $\Omega$ satisfy $\rho(E,u,\Omega)\leq \gamma u^q$, $u>0$, where $1<q\leq 2$.
If for a given    $\e>0$, there is  an element  $\varphi^{\e}\in {\mathcal D}$, such that 
$$
E(\varphi^{\e})\leq \inf_{x \in \Omega}E(x)+\e, \quad \varphi^{\e}/A(\e) \in  A_1(\mathcal D),
$$
for some constant $A(\e)\geq 1$,
then, the output $x^{\rm w}_{m}$ of the WCGA  satisfies the inequality
$$E(x^{\rm w}_m)-\inf_{x\in \Omega}E(x) \leq\max\left\{2\e, C_1A(\e)^q\bigl(C_2+\Sigma_{k=1}^mt_k^{q/(q-1)}\bigr)^{1-q}\right\},$$
with constants  $C_1=C_1(q,\gamma)$ and $C_2=C_{2}(E,q,\gamma)$.
\end{theorem}

 \section{Main results}
 \label {AL}
 In this section, we present our main results and the auxiliary lemmas, needed for their proof.
 First, note that the set $\Omega:=\{x \in H: E(x)\leq E(0)\}$ is convex since it is the level set of a convex function. Also,  
 all outputs $\{x_k\}_{k=1}^{\infty}$ generated by the  OMP(co) (or the  WCGA(co)) are in $\Omega$, since the sequence
$\{E(x_k)\}_{k=1}^{\infty}$ is decreasing and $E(x_1)\leq E(0)$. 
 
\subsection{Auxiliary lemmas} 

\noindent 

Here, we begin with some lemmas that we use to derive our main results. The next lemma is well-known.
  \begin{lemma}
\label{lm0}  
Let $F$ be a Frechet differentiable function. 
Let $V_k:=\text{span}\{\varphi_{j_1}, \dots,\varphi_{j_k}\}$ and $x_k:=\argmin \{F(x):{x\in V_k}\}$. 
Then, we have that $\langle F'( x_k),\varphi\rangle =0$ for every $\varphi \in V_k$.
\end{lemma}

Our  next lemma can be viewed as a  generalization  of Lemma 2.16 from \cite{Tbook}.
\begin{lemma}
\label{lmseq}
 Let $\ell>0$, $r>0$, $B>0$,  $\{a_m\}_{m=1}^{\infty}$ and  $\{r_m\}_{m=2}^{\infty}$ be sequences of non-negative numbers satisfying the inequalities
$$ a_1\leq B, \quad a_{m+1} \leq a_m(1- \frac{r_{m+1}}{r}a_m^\ell), \quad m=1,2,\dots.$$
Then, we have 
\begin{equation}
\label{tuti12}
 a_m \leq \max\{1,\ell^{-1/\ell}\}r^{1/\ell}(rB^{-\ell}+\Sigma_{k=2}^m r_k)^{-1/\ell}, \quad m=2,3, \ldots.
 \end{equation}
\end{lemma}
\begin{proof} 
Let us first notice that from the recursive relation and the fact that all $a_m$'s are non-negative, we have 
\begin{equation}
\label{ppp}
0\leq1- \frac{r_{m+1}}{r}a_m^\ell\leq1, \quad m=1,2, \ldots.
\end{equation}

 We will show that for $m=2,3, \ldots$
\begin{equation}
\label{tuti}
 a_m^\ell \leq 
\left
\{\begin{array}{cc}
\displaystyle{\frac {r}{(rB^{-\ell}+\Sigma_{k=2}^m r_k)}}, & \mbox{     if  } \,\,\ell\geq 1, \\\\
\displaystyle{\frac {r}{(rB^{-\ell}+\ell\Sigma_{k=2}^m r_k)}}, &\mbox{if }\,\,0<\ell\leq  1,
\end{array}
\right.
\end{equation}
from which the inequality \eref{tuti12} easily follows.

\noindent
We prove \eref{tuti}  by induction.

\noindent
{\bf Case 1: $\ell\geq 1$.}

If $a_2=0$, then all $a_m=0$, $m=3,4,\ldots$, and the lemma is true. Let us assume that $a_2>0$, and therefore  $a_1>0$.
It follows from the recursive relation and  \eref{ppp} that for $\ell\geq 1$ 
$$
a_{2} ^{-\ell}\geq a_1^{-\ell}(1-\frac{r_{2}}{r}a_1^\ell)^{-\ell}\geq a_1^{-\ell}(1-\frac{r_{2}}{r}a_1^\ell)^{-1} \geq a_1^{-\ell}(1+\frac{r_{2}}{r}a_1^\ell)=a_1^{-\ell}+\frac {r_{2}} r\geq B^{-\ell}+\frac {r_{2}} r.
$$ 
This gives
\eref{tuti} for $m=2$.
 
We now assume that \eref{tuti} is true for $m$ and  prove it's validity  for $m+1$.
 As in  the case $m=2$, we may assume that  $a_{m+1}>0$. Because of the recursive relation, this also means that $a_m>0$ and using \eref{ppp},
 we derive
\begin{equation}
\label{po}
a_{m+1} ^{-\ell}\geq a_m^{-\ell}(1-\frac{r_{m+1}}{r}a_m^\ell)^{-\ell}  
 \geq 
a_m^{-\ell}(1+\frac{r_{m+1}}{r}a_m^\ell)=a_m^{-\ell}+\frac {r_{m+1}} r. 
\end{equation}
Now, from  the induction hypothesis we have that
$$
a_m^{-\ell}\geq \frac {rB^{-\ell}+ \Sigma_{k=2}^m r_k}r,
$$
which combined with \eref{po} proves the lemma in the case $\ell\geq 1$.

\noindent
{\bf Case 2: $0<\ell < 1$.} 

Again, we need only consider the case when $a_2>0$. We will use the fact that for  $0<\ell  < 1$, the function $(1-t)^\ell$ is concave.  Therefore,  we have
\begin{equation}
\label{ola}
(1-t)^\ell\leq 1-\ell t,\quad 0\le t\le 1.
\end{equation}
We apply this inequality with  $t=\frac{r_{2}}{r}a_1^\ell\in [0,1]$
and obtain
\begin{eqnarray}
\nonumber
a_{2} ^{-\ell}&\geq& a_1^{-\ell}(1-\frac{r_{2}}{r}a_1^\ell)^{-\ell}\geq a_1^{-\ell}(1-\ell\frac{r_{2}}{r}a_1^\ell)^{-1} \geq a_1^{-\ell}(1+\ell\frac{r_{2}}{r}a_1^\ell)\\ \nonumber
&=&a_1^{-\ell}+\ell\frac {r_{2}} r\geq B^{-\ell}+\ell\frac {r_{2}} r,
\end{eqnarray}
which gives  \eref{tuti} for $m=2$. Next , we assume that \eref{tuti} is true for $m$ and   prove it for $m+1$.
We can assume $a_{m+1}>0$  and therefore $a_m>0$. From the recursive relation and \eref{ola} 
with $t=\frac{r_{m+1}}{r}a_m^\ell\in [0,1]$, we have
\begin{eqnarray}
\nonumber
a_{m+1} ^{-\ell}&\geq& a_m^{-\ell}(1-\frac{r_{m+1}}{r}a_m^\ell)^{-\ell} \geq a_m^{-\ell}(1-\ell\frac{r_{m+1}}{r}a_m^\ell)^{-1}\\ \nonumber
&\geq& 
a_m^{-\ell}(1+\ell\frac{r_{m+1}}{r}a_m^\ell)=a_m^{-\ell}+\ell\frac {r_{m+1}} r. 
\end{eqnarray}
This inequality, combined with the induction hypothesis gives that
$$
a_{m+1}^{-\ell}\geq \frac {rB^{-\ell}+ \ell\Sigma_{k=2}^{m+1} r_k}r,
$$
and the proof is complete.
\end{proof}

 \subsection{Convergence rates for OMP(co)}
 \label {MROMP}
 
 \noindent

 In this section, we analyze the performance of the OMP(co) algorithm when applied to the minimization problem (\ref{Eq})
 with $D=H$. 
 We assume that  the dictionary $\mathcal D$ is an orthonormal system $\{\varphi_i\}_{i=1}^\infty$ and $E$ takes on
 its global minimum $\bar x$.  This means that this global minimum is assumed over $\Omega:=\{x: \ E(x)\le E(0)\} $. 
Let us denote by $e_k$ the error of the algorithm at Step k, namely,
$$
e_k:=E(x_k)-E(\bar x).
$$
The next lemma provides a recursive relation for the sequence $\{e_k\}_{k=1}^\infty$. 
\begin{lemma}
\label {L1}
Let the  objective function $E$ satisfy {\bf Conditions 0}, {\bf 1}, and {\bf 2}, and $\mu$ be a  constant such that 
$\mu>\max\{1,M_0\alpha^{-1}M^{1-q}\}$.  Let problem   {\rm (\ref{Eq}) } have a solution  
$\bar x=\sum_{i}c_i(\bar x)\varphi_i\in \Omega$ with support 
$\bar S:=\{i:\, c_i(\bar x)\neq 0\}<\infty$,
where $\{\varphi_i\}$ is an orthonormal basis.
Then,    the error of the OMP(co)  applied  to $E$   and $\{\varphi_i\}$ satisfies the following recursive inequalities:
\begin{equation}
\label {basic0}
e_1 \leq E(0)-E(\bar x),
\end{equation}
and  
\begin{equation}
\label {basic1}
e_k \leq e_{k-1} -\frac{(\mu-1)\mu^{-q/(q-1)}}{r}e_{k-1}^{\frac{(p-1)q}{(q-1)p}},\quad k\ge 2,
\end{equation}
where the constant $r$ is 
$$
r=|\bar S|^\frac{q}{2(q-1)}\alpha^{\frac{1}{q-1}}\left (p\be_0^{1/p}(p-1)^{(1-p)/p}\right )^{-q/(q-1)}.
$$
\end{lemma}
\noindent
\begin{proof}
Clearly, we have  $e_1=E(x_1)-E(\bar x)\leq E(0)-E(\bar x)$
since 
$${\displaystyle x_1:=\displaystyle {\argmin_{x\in span\{\varphi_{j_1}\}}  E(x)}}.$$
Next, 
we consider  Step k, $k=2,3,\ldots$ of the algorithm.
Observe that if at Step (k-1) we have that $ \bar S \subseteq\{j_1, \ldots, j_{k-1}\}$, then $x_{k-1}=\bar x$, $E'(x_{k-1})=0$ and the 
OMP(co) would have stopped with output $x_{k-1}=\bar x$. If the algorithm  has  not stopped, then
 it generates the next output $x_k$ and $\varphi_{j_k}$. Since $x_k$ is the point of  minimum of $E$ over $ span\{\varphi_{j_1}, \varphi_{j_2},\dots,\varphi_{j_k}\}$, we have for any $|t|\le M$,
\begin{equation}
\label{q}
E(x_k) \leq E(x_{k-1}+t \varphi_{j_k})\leq E(x_{k-1})  +t\ \lan E'(x_{k-1}),\varphi_{j_k}\ran+\al |t|^q,
\end{equation}
where the last inequality invoked 
 {\bf Condition 1}.   We take 
 $$t=-\left(\alpha \mu\right)^{-\frac{1}{q-1}} {\rm sign}(\lan E'(x_{k-1}),\varphi_{j_k}\ran) |\lan E'(x_{k-1}),\varphi_{j_k}\ran|^{\frac{1}{q-1}}.
 $$
 Because of the definition of $\mu$ in the statement of the theorem, we have $|t|\le M$.  Therefore,
 we have
\begin{equation}
\label{q11}
E(x_k) \leq E(x_{k-1}) -\frac{\mu-1}{\mu}\left(\alpha \mu\right)^{-\frac{1}{q-1}}\abs{\lan E'(x_{k-1}),\varphi_{j_k}\ran}^{q/(q-1)}.
\end{equation}

Now,  we will find a lower bound for $\abs{\lan E'(x_{k-1}),\varphi_{j_k}\ran}$.
First, note that from {\bf Condition 2} and Claim 1 applied to $x'=\bar x$ and $x=x_{k-1}$ (both are in $\Omega$), we obtain 
\begin{equation}
\label{pp}
\lan E'(x_{k-1}), x_{k-1}-\bar x \ran \geq E(x_{k-1})-E(\bar x)+\be_0 \norm{\bar x-x_{k-1}}^p.
\end{equation}
Let us recall the weighted arithmetic mean -geometric mean inequality
$$
\frac{p_1}{p_1+p_2}a+\frac{p_2}{p_1+p_2}b\geq a^\frac{p_1}{p_1+p_2}b^\frac{p_2}{p_1+p_2},\quad  \hbox{where}\quad a,b\geq 0, \quad p_1,p_2>0,
$$
and apply it for $p_1=p-1$, $p_2=1$, $a=\frac{E(x_{k-1})-E(\bar x)}{p-1}$, $b=\be_0 \norm{\bar x-x_{k-1}}^p$.
We have
 \begin{eqnarray}
 \nonumber
E(x_{k-1})-E(\bar x)+\be_0 \norm{\bar x-x_{k-1}}^p=
p\left (\frac {(p-1)}{p}\frac{E(x_{k-1})-E(\bar x)}{p-1}+ \frac{1}{p}\be_0 \norm{\bar x-x_{k-1}}^p\right ),
\end{eqnarray}
and therefore
\begin{eqnarray}
 \nonumber
E(x_{k-1})-E(\bar x)+\be_0 \norm{\bar x-x_{k-1}}^p
\geq C\norm{\bar x-x_{k-1}}\left(E(x_{k-1})-E(\bar x)\right)^{(p-1)/p}, 
\end{eqnarray}
with $C=p\be_0^{1/p}(p-1)^{(1-p)/p}$. 
We combine this inequality with \eref{pp} to obtain
\begin{equation}
\label{ppnew}
\lan E'(x_{k-1}), x_{k-1}-\bar x \ran \geq C\norm{\bar x-x_{k-1}}\left(E(x_{k-1})-E(\bar x)\right)^{(p-1)/p}.
\end{equation}

From the definition of $x_{k-1}$ and Lemma \ref{lm0}, it follows  that 
$$\lan E'(x_{k-1}), \varphi_i\ran=0, \quad i=j_1,\dots,j_{k-1}.$$
 Therefore, if we write 
 $$
 x_{k-1}-\bar x=\sum_{i}c_i(x_{k-1}-\bar x)\varphi_i,
 $$
since the support of $x_{k-1}$ is $\{j_1,\ldots,j_{k-1}\}$, 
 we obtain
\begin{eqnarray}
\label{ta}
\nonumber
 \lan E'(x_{k-1}), x_{k-1}-\bar x \ran &=&  \sum_{i \in   \bar S \setminus \{ j_1,\ldots, j_{k-1}\}} c_i(x_{k-1}-\bar x)\lan E'(x_{k-1}), \varphi_i \ran,\nonumber\\
  &\leq& \sum_{i \in   \bar S \setminus \{j_1,\ldots, j_{k-1}\}}\abs{c_i(x_{k-1}-\bar x)} \abs{\lan E'(x_{k-1}),\varphi_{j_k}\ran}\\
\nonumber
&\leq &\abs{\lan E'(x_{k-1}),\varphi_{j_k}\ran} |\bar S |^{1/2}\|x_{k-1}-\bar x\|. 
\nonumber
 \end{eqnarray}
 We combine this  inequality with   \eref{ppnew} and derive that 
\begin{equation}
\nonumber
\abs{\lan E'(x_{k-1}),\varphi_{j_k}\ran} \|\bar x-x_{k-1}\||\bar S|^{1/2}\geq C\norm{\bar x-x_{k-1}}\left(E(x_{k-1})-E(\bar x)\right)^{(p-1)/p}.
\end{equation}
Therefore we have the desired lower bound
$$
\abs{\lan E'(x_{k-1}),\varphi_{j_k}\ran} \geq C|\bar S|^{-1/2}\left(E(x_{k-1})-E(\bar x)\right)^{(p-1)/p}.
$$
The latter  result and  \eref{q11} gives the estimate
$$
E(x_k) \leq E(x_{k-1}) -\frac{(\mu-1)C^{q/(q-1)}}{\mu^{q/(q-1)}\alpha^{1/(q-1)}|\bar S|^{\frac{q}{2(q-1)}}}\left(E(x_{k-1})-E(\bar x)\right)^{\frac{(p-1)q}{(q-1)p}}.
$$
Subtracting $E(\bar x)$ from both sides of this inequality results in \eref{basic1} and the proof is completed.
\end{proof}

We next remark that we can  take a specific value for $\mu$ in the last lemma.
\begin{rem}
\label{const}
Let the  objective function $E$ satisfy {\bf Conditions 0}, {\bf 1}, and {\bf 2}.  Let problem   {\rm (\ref{Eq}) } have a solution  
$\bar x=\sum_{i}c_i(\bar x)\varphi_i\in \Omega$ with support 
$\bar S:=\{i:\, c_i(\bar x)\neq 0\}<\infty$,
where $\{\varphi_i\}$ is an orthonormal basis.
Then,    the error of the OMP(co) applied  to $E$   and $\{\varphi_i\}$ satisfies the following recursive inequalities:
$$
e_1 \leq E(0)-E(\bar x),
$$
and  
\begin{equation}
\label {basicrem}
e_k \leq e_{k-1} -\frac{C_3}{r}e_{k-1}^{\frac{(p-1)q}{(q-1)p}}=e_{k-1}[1 -\frac{C_3}{r}e_{k-1}^{\frac{p-q}{(q-1)p}}],\quad k\ge 2,
\end{equation}
where $r$ is the constant from Lemma {\rm \ref{L1}}
and $C_3=C_3(M_0,M, \alpha,q)$ is 
\begin{equation}
\label{c3}
C_3=
\left
\{\begin{array}{cc}
\displaystyle{(q-1)q^{-q/(q-1)}}, & \mbox{     if  } \,\,M_0M^{1-q}\alpha^{-1}< q, \\\\
\displaystyle{ (M_0M^{1-q}\alpha^{-1}-1)M_0^{-q/(q-1)}M^{-q}\alpha^{q/(q-1)}}, &\mbox{if }\,\,M_0M^{1-q}\alpha^{-1}\geq q.
\end{array}
\right.
\end{equation}
\end{rem}
\begin{proof}
The estimate follows from Lemma {\rm \ref{L1}} and the fact that the function 
$$g(\mu)=(\mu-1)\mu^{-q/(q-1)}$$
 is increasing on $(1,q)$ and 
decreasing on $(q,\infty)$ with global maximum at $\mu=q$.
\end{proof}

The next theorem is our  main result about the OMP(co) algorithm.
 \begin{theorem}
\label{Tm0}
Let the  objective function $E$ satisfy {\bf Conditions 0} ,{\bf 1}, and {\bf 2}.  Let  problem 
{\rm (\ref{Eq})} with $D=\Omega:=\{x:\ E(x)\le E(0)\}$ have a solution  
$\bar x=\sum_{i}c_i(\bar x)\varphi_i\in \Omega$ with support  $\bar S:=\{i:\, c_i(\bar x)\neq 0\}<\infty,$ where 
$\{\varphi_i\}$ is an orthonormal basis for $H$.
Then, at Step k, the OMP(co) applied to $E$ and $\{\varphi_i\}$ outputs $x_k$, where either $x_k=\bar x$, in which case    $e_k=0$, or:

\noindent
{\rm (i)} When  $p\neq q$, for $k=2,3,\ldots,$

$$
e_k\leq C|\bar S|^{\frac{pq}{2(p-q)}}k^{-\frac{p(q-1)}{p-q}},
$$
$$
 \|x_k-\bar x\|\leq C'|\bar S|^{\frac{q}{2(p-q)}}k^{-\frac{q-1}{p-q}}
 $$
where $C$ and $C'$ depend only on $p,q,\alpha,\beta,E$.

 \noindent
 {\rm(ii)} When $p=q=2$, we have the exponential decay
$$
e_k\leq C_2\gamma ^{k-1},
$$
$$
\|x_k-\bar x\|\leq C_2^{\frac{1}{2}}\beta_0 ^{-\frac{1}{2}}\gamma^{(k-1)/2},
\quad k=2,3,\ldots,
$$
where     $ \gamma:=  1-\frac{\tilde C_3}{|\bar S|}$ is in $(0,1)$,   $C_2=E(0)-E(\bar x)$, and $\tilde C_3$ is a constant 
that depends on $\alpha$,  $\beta$,  and $E$.
 \end{theorem}
\begin{proof}
In the case $p\neq q$, we define the sequence of non-negative numbers
$$
r_k=C_3, \quad a_k=E(x_k)-E(\bar x), \quad k=1,2,\ldots,
$$
and the numbers
$$
r=|\bar S|^\frac{q}{2(q-1)}\alpha^{\frac{1}{q-1}}\left (p\be^{1/p}(p-1)^{(1-p)/p}\right )^{-q/(q-1)}>0, 
$$
$$
\ell=\frac{p-q}{p(q-1)}>0, \quad B=E(0)-E(\bar x)> 0.
$$
It follows from Remark \ref{const} that the above defined sequences satisfy the conditions of Lemma \ref{lmseq}, and therefore
we have
\begin{equation}
\label{ekest}
e_k=E(x_k)-E(\bar x)\leq C_0\left (\frac{|\bar S|^{\frac{q}{2(q-1)}}}{C_1|\bar S|^{\frac{q}{2(q-1)}}+C_3(k-1)}\right )^{\frac{p(q-1)}{p-q}},
\end{equation}
where
$$
C_0=C_0(p,q,\alpha,\beta)=\alpha^{\frac{p}{p-q}}\left (p\be^{1/p}(p-1)^{(1-p)/p}\right )^{-\frac{pq}{p-q}}\cdot\max\left\{1,\left (\frac{p(q-1)}{p-q}\right)^{\frac{p(q-1)}{p-q}}\right \}, 
$$ and 
$$
C_1=C_1(p,q,\alpha,\beta,E)=\alpha^{\frac{1}{q-1}}\left (p\be^{1/p}(p-1)^{(1-p)/p}\right )^{-q/(q-1)}
(E(0)-E(\bar x))^
{\frac{q-p}{p(q-1)}}.
$$
One easily derives the estimate for $e_k$ in (i) from \eref{ekest}.  The estimate for $\|x_k-\bar x\|$ in (i)
now follows from  {\bf Condition 2} with $x'=x_k$, $x=\bar x$ and   Lemma \ref{FT} .

In the case $p=q=2$, as before  $E(x_1)-E(\bar x)\leq E(0)-E(\bar x)$,  and Lemma \ref{L1}  and Remark \ref{const}
give that
$$
E(x_k)-E(\bar x)\leq \left (1-\frac{\tilde C_3}{|\bar S|}\right )(E(x_{k-1})-E(\bar x)), \quad k=2,3,\ldots,
$$
where
\begin{equation}
\label{c31}
\tilde C_3=
\left
\{\begin{array}{cc}
\displaystyle{\frac{\beta_0}{\alpha}}, & \mbox{     if  } \,\,M_0M^{-1}\alpha^{-1}< 2, \\\\
\displaystyle{ 4\beta_0(M_0M^{-1}\alpha^{-1}-1)M_0^{-2}M^{-2}\alpha}, &\mbox{if }\,\,M_0M^{-1}\alpha^{-1}\geq 2.
\end{array}
\right.
\end{equation}
It follows  that
$$
E(x_k)-E(\bar x)\leq (E(0)-E(\bar x))\left (1-\frac{\tilde C_3}{|\bar S| }\right )^{k-1}, \quad k=2,3,\ldots.
$$
As in the previous case, we use {\bf Condition 2} with $x'=x_k$, $x=\bar x$ and  Lemma \ref{FT} to derive the estimate for $\|x_k-\bar x\|$.
\end{proof}

 \subsection{Main results for WCGA(co)}
 \label {wMROMP}
 The convergence analysis of the WCGA(co) is almost the same as that for  the OMP(co). We omit the details here and just state the error 
 estimates for 
 $$
e^{\rm w}_k:=E(x^{\rm w}_k)-E(\bar x),
$$
pointing out the main differences in the proof.
 \begin{theorem}
 \label{wlm0}
Let the  objective function $E$ satisfy {\bf Conditions 0}, {\bf 1}, and {\bf 2}. Let  problem {\rm (\ref{Eq})} with $D=\Omega=\{x:\ E(x)\le E(0)\}$ have a solution  
$\bar x=\sum_{i}c_i(\bar x)\varphi_i\in \Omega$ with support  $\bar S:=\{i:\, c_i(\bar x)\neq 0\}<\infty$, where 
$\{\varphi_i\}$ is an orthonormal basis. 
Then, at Step k, the WCGA applied to $E$ and $\{\varphi_i\}$ outputs $x^{\rm w}_k$, where either $x^{\rm w}_k=\bar x$,
in which case $e^{\rm w}_k=0$,  or:

\vskip .1in
\noindent
{\rm (i)} When  $p\neq q $, for  each $k=2,3,\ldots$, we have
$$
e^{\rm w}_k\leq \tilde C|\bar S|^{\frac{pq}{2(p-q)}}\left (\sum_{j=2}^kt_j^{\frac{q}{q-1}}\right )^{\frac{p(q-1)}{p-q}}
$$
$$
 \|x^{\rm w}_k-\bar x\|\leq \tilde C'|\bar S|^{\frac{q}{2(p-q)}}\left (\sum_{j=2}^kt_j^{\frac{q}{q-1}}\right )^{\frac{(q-1)}{p-q}}
 $$
where
$\tilde C$ and $\tilde C'$ depend only on $p$, $q$, $\alpha$, $\beta$, $E$.

\vskip .1in
\noindent
{\rm (ii)}  When $p=q=2$, we have
$$
e_k^{\rm w}\leq C_2\prod_{j=2}^k\left (1-\frac{\tilde C_3}{|\bar S|}t_j^2\right ),
$$
$$
\|x_k^{\rm w}-\bar x\|\leq C_2^{\frac{1}{2}}\beta^{-\frac{1}{2}}\prod_{j=2}^k\left (1-\frac{\tilde C_3}{|\bar S|}t_j^2\right )^{1/2},
$$
with $C_2=E(0)-E(\bar x)$ and $\tilde C_3$ depends on $\alpha$, $\beta$, and $E$.
 \end{theorem}
\begin{proof} 
The proof follows the lines of that of Theorem \ref{Tm0} and the corresponding lemmas. The difference is that instead of estimate \eref{ta}, we have
\begin{eqnarray}
\nonumber
 \lan E'(x^{\rm w}_{k-1}), x^{\rm w}_{k-1}-\bar x \ran &=& \Sigma_{i \in \bar S \setminus j_1,\ldots, j_{k-1}} c_i(x^{\rm w}_{k-1}-\bar x)\lan E'(x^{\rm w}_{k-1}), \varphi_i \ran\\
\nonumber
& \leq& \Sigma_{i \in \bar S \setminus j_1,\ldots, j_{k-1}}\abs{c_i(x^{\rm w}_{k-1}-\bar x)} \abs{\lan E'(x^{\rm w}_{k-1}),\varphi_i\ran}\\
\nonumber
&\leq &t_k^{-1}\abs {\lan E'(x^{\rm w}_{k-1}),\varphi_{j_k}\ran} \Sigma_{i \in \bar S}\abs{c_i(x^{\rm w}_{k-1}-\bar x)}\\
&\leq&t_k^{-1}\abs {\lan E'(x^{\rm w}_{k-1}),\varphi_{j_k}\ran}\|\bar x-x^{\rm w}_{k-1}\||\bar S|^{1/2},
\label{wta}
 \end{eqnarray}
 and that we use Lemma \ref{lmseq} with 
$r_k=C_3{t_k}^{\frac{q}{q-1}}$, $k=2,3,\ldots$.
\end{proof}

 \vskip .1in
  
  \noindent
 Hao Nguyen\\ 
 Department of Mathematics, Texas A\&M University,
College Station, TX 77843, USA\\
htnguyen@math.tamu.edu 
  
    \vskip .1in
 \noindent
Guergana Petrova\\
Department of Mathematics, Texas A\&M University,
College Station, TX 77843, USA\\
  gpetrova@math.tamu.edu 
\end{document}